\documentclass[12pt]{amsart}
\usepackage{amsthm,amssymb,verbatim,color}

\numberwithin{equation}{section}
\newtheorem{thm}{Theorem}[section]

\newtheorem{lem}[thm]{Lemma}

\newtheorem{problem}[thm]{Question}

 \theoremstyle{definition}
 \newtheorem{defn}[thm]{Definition}
 \theoremstyle{remark}
 \newtheorem{rem}[thm]{Remark}
 \theoremstyle{definition}

\newcommand{\CC}{\mathbb{C}}

\newcommand{\XX}{\mathbb{X}}
\newcommand{\PP}{\mathbb{P}}
\newcommand{\pr}{\mathbb{P}}

\begin{document}

\title{Star configurations on generic hypersurfaces}

\author[E. Carlini]{Enrico Carlini}
\address[E. Carlini]{Dipartimento di Scienze Matematiche,
Politecnico di Torino, Turin, Italy}
\email{enrico.carlini@polito.it}

\author[E. Guardo]{Elena Guardo}
\address[E. Guardo]{Dipartimento di Matematica,
Universit\`a di Catania, Catania, Italy}
\email{guardo@dmi.unict.it}

\author[A. Van Tuyl]{Adam Van Tuyl}
\address[A. Van Tuyl]{Department of Mathematical Sciences,
Lakehead University, Thunder Bay, ON, Canada, P7B 5E1}
\email{avantuyl@lakeheadu.ca}


\thanks{Version: April 10, 2012}


\begin{abstract}
Let $F$ be a homogeneous polynomial in $S =
\mathbb{C}[x_0,\ldots,x_n]$.  Our goal is to understand a
particular polynomial decomposition of $F$; geometrically, we wish
to determine when the hypersurface defined by $F$ in
$\mathbb{P}^n$ contains a star configuration. To solve this
problem, we use techniques from commutative algebra and algebraic
geometry to reduce our question to computing the rank of a matrix.
\end{abstract}

\maketitle

\begin{center}
{\it To A.V. Geramita on the occasion of his 70th birthday.}
\end{center}

\section{Introduction}

Throughout this paper we work over the polynomial ring
$S=\mathbb{C}[x_0,\ldots,x_n]$.  Given any homogeneous polynomial
$F \in S$ of degree $d$, one usually writes $F$ as a sum of
monomials of degree $d$, i.e., $F = \sum c_im_i$, where $c_i \in
\mathbb{C}$ and $m_i$ is a monomial of degree $d$. However,
different presentations are possible;  for example, one can look
for a sum of powers presentation of $F$, that is, find linear
forms $\ell_1,\ldots,\ell_k$ such that
\[F = c_1\ell_1^d + c_2\ell_2^d + \cdots + c_k\ell_k^d.\]
Given a possible presentation, one can then ask many relevant
questions about the presentation.  Two such problems would be to
find the minimal number of summands needed for the generic form,
or for any given form, find an explicit presentation. Questions of
this type were explored in the work \cite{Ca04JA,Ci01}.

Presentations of $F$ can also be reinterpreted as geometric
questions.  As an example, if $F$ has a sum of powers
presentation as above, then $F$ is an element of the ideal $I =
(\ell_1,\ldots,\ell_k)$.  But this means that the hypersurface
defined by $F$ in $\mathbb{P}^n$ contains the variety defined by
$I$.  A presentation question could therefore be reformulated as
asking if a (generic) hypersurface contains a special subvariety.
This type of question has a long history, e.g.,
a number of authors have investigated the question of
when a hypersurface contains a complete intersection (the papers
\cite{CCG1,CCG2,Gro05,Lef21,Sev06,Szabo} form a partial list of
papers devoted to this topic).

In this paper, we want to investigate the following type of
polynomial decomposition.   To state our question, we use the
notation $[l] = \{1,\ldots,l\}$.

\begin{problem}\label{question}
For which tuples $(n,l,r,d) \in \mathbb{N}_+^4$ is it possible to
present a generic homogeneous form $F$ of degree $d$ in $n+1$
variables as
\begin{equation}\label{expression}
F=\sum_{
\footnotesize{
\begin{array}{c}
\sigma = \{i_1,\ldots,i_r\} \subseteq [l],\\
 |\sigma| = r
\end{array}}}
L_\sigma M_\sigma
\end{equation}
where $L_{\sigma} = L_{i_1}L_{i_2} \cdots L_{i_r}$,
$\{L_1,\ldots,L_l\}$ are generic linear forms, and  $M_{\sigma}$ is a form
of degree $d-r$.
\end{problem}

In order to express $F$ in the form \eqref{expression}, we immediately notice
some simple restrictions on the tuples $(n,l,r,d)$, namely $r \leq l$ and $r \leq d$.
The goal of this paper is to give an almost complete answer to this question.
Our main result is:

\begin{thm}\label{mainresult}
Let $(n,l,r,d) \in \mathbb{N}_+^4$ be such that $r\leq \min\{d,l\}$.
\begin{enumerate}
\item  If $l-r+1 < n$ and $d\gg 0$, then the generic degree $d$ form
in $n+1$ variables cannot
be written in the form \eqref{expression}.
\item If $l-r+1=n$, then the generic degree $d$ form in $n+1$ variables can
be written in the form \eqref{expression} if and only if
$(n,l,r,d)$ belongs to the following list:
\begin{enumerate}
\item[$(i)$] $(n,l,r,d) = (1,l,l,d)$ for all $d \geq l \geq 1$, or
\item[$(ii)$] $(n,l,r,d) = (2,2,1,d)$ for all $d \geq 1$, or
\item[$(iii)$] $(n,l,r,d) = (2,3,2,d)$ for all $d \geq 2$, or
\item[$(iv)$] $(n,l,r,d) = (2,4,3,d)$ for all $d \geq 3$, or
\item[$(v)$] $(n,l,r,d) = (2,5,4,d)$ for all $d \geq 5$, or
\item[$(vi)$] $(n,l,r,d) = (n,n,1,d)$ for all $n \geq 3$ and $d \geq 1$, or
\item[$(vii)$] $(n,l,r,d) = (n,n+1,2,d)$ for all $n \geq 3$ and $d \geq 2$, or
\item[$(viii)$] $(n,l,r,d) = (n,n+2,3,d)$ for all $n \geq 3$ and $d \geq 3$.
\end{enumerate}
\item If $l-r +1 >n$, then every degree $d$ form (not just the generic
one) in $n+1$ variables can be
written in the form \eqref{expression}.
\end{enumerate}
\end{thm}

Geometrically, Question \ref{question} is asking if the
generic degree $d$ hypersurface contains a star configuration (see
the definition in the next section).  Question \ref{question} was studied in
the case that $n=2$ by the first and third author in \cite{CVT}.  We
refer the reader to this paper for the statements of
 Theorem \ref{mainresult} involving  $n=2$.  Note that
our answer is almost complete since there may be tuples
$(n,l,r,d)$ with $l-r+1 <n$ with $d$ small enough such that Question
\ref{question} has a positive answer.  However, we currently
know of no such examples.

Our paper is structured as follows.  In the next section, we give
two interpretations of Question \ref{question}:  an algebraic
version and a geometric version.  The geometric version of this
question asks about star configurations on hypersurfaces.  We also
prove some cases of Theorem \ref{mainresult}.  In Section 3, we
look at non-existence results, that is, look for ways to eliminate
various $(n,l,r,d)$ from consideration.  By the end of Section 3,
we will have proved all of Theorem \ref{mainresult} except the
part of statement $(2)$ involving the tuples $(n,n+2,3,d)$.
The remainder of the paper is devoted to the proof of this case.
In Section 4, we translate our question again.  The new
translation reduces our question to showing that a specific
evaluation matrix has maximal rank.  We then answer this
corresponding linear algebra question in Sections 5 and 6.

{\bf Acknowledgements}  We would like to thank A.V. Geramita for
answering some of our questions while preparing this paper and
sharing with us \cite{GHM}.   The first author acknowledges the
support of the Politecnico di Torino and of the University of
Catania. The third author acknowledges the support of NSERC.


\section{Star Configurations}

We reformulate Question \ref{question} as an algebraic question
and a geometric question.   To state the geometric counterpart, we
will introduce star configurations.

We begin with the algebraic reformulation.
In $S = \mathbb{C}[x_0,\ldots,x_n]$, let
$L_1,\ldots,L_l$ be $l$ linear homogeneous forms.
We let $[l] = \{1,\ldots,l\}$
and we set
\[L_{\sigma} := L_{i_1}L_{i_2} \cdots L_{i_r} ~~\mbox{for any $\sigma = \{i_1,\ldots,i_r\} \subseteq [l]$.} \]
We shall write $V(L)$ to mean the hypersurface in $\PP^n$ defined
by $L$. If $L$ is linear, then $V(L)$ is usually called a {\bf
hyperplane}. We say that the $l$ linear homogeneous forms
$L_1,\ldots,L_l$ are {\bf general linear forms} if any $n+1$ of
the linear forms are linearly independent. If $l < n+1$, then we
require that the $l$ linear forms are linearly independent.

The algebraic reformulation of Question \ref{question} is an ideal
membership problem.

\begin{problem}[Algebraic Question]\label{algebraQ}
Fix a tuple $(n,l,r,d) \in \mathbb{N}_+^4$ with $r \leq \min\{d,l\}$. Given a generic homogeneous form $F \in S =
\mathbb{C}[x_0,\ldots,x_n]$ of degree $d$, is it possible to find
$l$ general linear forms $L_1,\ldots,L_l$ such that
$F \in I = (L_\sigma ~|~  \sigma \subseteq [l] ~~\mbox{and}~ |\sigma| = r )?$
\end{problem}

The geometric interpretation of Question \ref{question} is in terms of star
configurations.

\begin{defn}  Let $L_1,\ldots,L_l$ be $l$ general linear forms
in $S = \mathbb{C}[x_0,\ldots,x_n]$.  Let $r \leq l$ be any positive
integer.  The {\bf star configuration} of type $(l,r)$, denoted
$\XX(l,r)$, is the algebraic variety of $\PP^n$
defined by the homogeneous ideal
\[J = \bigcap_{
\footnotesize{
\begin{array}{c}
\tau = \{j_1,\ldots,j_{l-r+1}\} \subseteq [l] \\
|\tau| = l-r+1
\end{array}
}}
(L_{j_1},\ldots,L_{j_{l-r+1}}).\]
Equivalently, the algebraic variety $\XX(l,r) = V(J)\subset\mathbb{P}^n$
is the union of all the linear spaces obtained by intersecting $l-r+1$ of the
hyperplanes $\{ L_i=0 \}$ in all possible ways.
\end{defn}

The  name ``star configuration'' was first suggested by A.V.
Geramita because a star configuration $\XX(5,4) \subseteq \PP^2$
resembles a star drawn with five lines.  Star configurations have
proven to be interesting varieties, in part, because they exhibit
some nice extremal behavior.  To date, much of the research
 (see
\cite{AS,BCH,CVT,CHT,GHM,GMS,HH,Shin1})
 has
focused on the case of star configurations of the type $(l,l-n+1)$;
in this case, $\XX(l,l-n+1)$ is a finite set of points.  This fact,
and others, will follow from the next
lemma which recalls some of the relevant properties of star
configurations for this project.

\begin{lem}\label{productintersectionlemma}
Let $L_1,\ldots,L_l$ be $l$ general linear forms of
$S = \mathbb{C}[x_0,\ldots,x_n]$ and $0 < r \leq l$.
\begin{enumerate}
\item[$(i)$]  If $l -r + 1 > n$, then $\XX(l,r) = \emptyset$.
\item[$(ii)$] If $l -r + 1 \leq n$, then $\dim \XX(l,r) = n -
(l-r+1)$. \item[$(iii)$] If $l-r+1 = n$, then $\XX(l,r)$ is a set
of $\binom{l}{n}$ distinct points. \item[$(iv)$] If $l -r + 1 \leq
n$, then $I_{\XX(l,r)} = (L_{\sigma} ~|~ \sigma \subseteq [l]
~\mbox{and}~ |\sigma| =r ).$
\end{enumerate}
\end{lem}

\begin{proof}
$(i)$  If $l -r +1 > n$, then for any $\tau \subseteq [l]$ with $|\tau| = l -r+1$,
the ideal $(L_{j_1},\ldots,L_{j_{l-r+1}})$ must be the irrelevant ideal
because the $L_i$'s are general linear forms.  Consequently
\[\XX(l,r) = V(J) =
\bigcup_{\footnotesize{
\begin{array}{c}
\tau \subseteq [l],~~|\tau| = l-r+1
\end{array}}}
V((L_{j_1},\ldots,L_{j_{l-r+1}})) = \emptyset.\]

$(ii)$  This fact follows directly from the definition of $J =
I_{\XX(l,r)}$ and from the fact that the $L_i$'s are general
linear forms.

$(iii)$  By $(ii)$, $\XX(l,r)$ is zero-dimensional.  For any $\tau
\subseteq [l]$ with $|\tau| = l -r + 1 = n$, the ideal
$(L_{j_1},\ldots,L_{j_{n}})$ defines a point in $\PP^n$.   There
are then $\binom{l}{n}$ such ideals, each defining a different
point.

$(iv)$  Let $I$ denote the ideal on the right
in the statement.  We first show $I \subseteq I_{\XX(l,r)}$.
Take any generator of $I$, say
$L_{\sigma}$ for some $\sigma \subseteq [l]$.  We claim that for
any subset $\tau = \{j_1,...,j_{l-r+1}\} \subseteq [l]$, the generator
$L_{\sigma} \in (L_{j_1},....,L_{j_{l-r+ 1}})$.  This claim follows
once we note that $\sigma \cap \tau \neq \emptyset$.  Indeed,
if these two sets were disjoint, then $|\sigma \cup \tau| =
r + (l-r+1) = l+1 > l$, which contradicts the fact that $[l]$ has
only $l$ distinct elements.   So, each generator of
$I$ belongs to $I_{\XX(l,r)}$, thus showing one inclusion.

For the reverse inclusion, we do induction on the tuple
$(l,r)$.  If $r = 1$ and for any integer $1= r \leq l$,
\begin{eqnarray*}
I_{\XX(l,1)} &=&  \bigcap_{
\footnotesize{
\begin{array}{c}
\tau \subseteq [l] \\
|\tau| = l-1+1
\end{array}
}}
(L_{j_1},\ldots,L_{j_{l-1+1}})
=(L_1,\ldots,L_l) =(L_\sigma ~|~ \sigma \subseteq [l] ~~\mbox{and} ~~|\sigma| = 1).
\end{eqnarray*}
In the case that $r=l$, we have
\begin{eqnarray*}
I_{\XX(l,l)} &=&  \bigcap_{
\footnotesize{
\begin{array}{c}
\tau \subseteq [l] \\
|\tau| = l-l+1 =1
\end{array}
}}
(L_{j_1},\ldots,L_{j_{l-l+1}}) =
(L_1) \cap (L_2) \cap \cdots \cap (L_l)
=(L_1\cdots L_l) \\
& = & (L_\sigma ~|~ \sigma \subseteq [l] ~~\mbox{and} ~~|\sigma| = l).
\end{eqnarray*}
So, statement $(iv)$ is true for all tuples of the form $(l,1)$
and $(l,l)$.

So, for the induction step, let $(l,r)$ be any tuple with
$1 < r < l$.  We then have
\footnotesize
\begin{eqnarray*}
I_{\XX(l,r)} &=&  \bigcap_{
\footnotesize{
\begin{array}{c}
\tau = \{j_1,\ldots,j_{l-r+1}\} \subseteq [l] \\
|\tau| = l-r+1
\end{array}
}}
(L_{j_1},\ldots,L_{j_{l-r+1}}) \\
&=&
 \bigcap_{
\footnotesize{
\begin{array}{c}
\tau = \{j_1,\ldots,j_{l-r+1}\} \subseteq [l] \\
|\tau| = l-r+1 ~~\mbox{and}~~ l \in \tau
\end{array}
}} (L_{j_1},\ldots,L_{j_{l-r+1}}) \cap
 \bigcap_{
\footnotesize{
\begin{array}{c}
\tau = \{j_1,\ldots,j_{l-r+1}\} \subseteq [l] \\
|\tau| = l-r+1  ~~\mbox{and}~~ l \not\in \tau
\end{array}
}}
(L_{j_1},\ldots,L_{j_{l-r+1}}) \\
& = &
 \bigcap_{
\footnotesize{
\begin{array}{c}
\tau = \{j_1,\ldots,j_{l-r-1+1}\} \subseteq [l-1] \\
|\tau| = (l-1)-r+1
\end{array}
}} (L_{j_1},\ldots,L_{j_{l-r-1+1}},L_l) \cap
 \bigcap_{
\footnotesize{
\begin{array}{c}
\tau = \{j_1,\ldots,j_{l-r+1}\} \subseteq [l-1] \\
|\tau| = (l-1)-(r-1)+1
\end{array}
}}
(L_{j_1},\ldots,L_{j_{l-r+1}}) \\
& = & (I_{\XX(l-1,r)},L_l) \cap I_{\XX(l-1,r-1)}.
\end{eqnarray*}
\normalsize
If we apply our induction hypothesis, we get
\begin{eqnarray*}
(I_{\XX(l-1,r)},L_l) \cap I_{\XX(l-1,r-1)} & = & ((L_{\sigma} ~|~ \sigma \subseteq [l-1] ~~\mbox{and}~~ |\sigma| = r),L_l) \cap (L_{\sigma} ~|~ \sigma \subseteq [l-1] ~~\mbox{and}~~ |\sigma| = r-1) \\
& \subseteq &
 (L_{\sigma} ~|~ \sigma \subseteq [l-1] ~~\mbox{and}~~ |\sigma| = r)
+ L_l(L_{\sigma} ~|~ \sigma \subseteq [l-1] ~~\mbox{and}~~ |\sigma| = r-1) \\
& = & (L_{\sigma} ~|~ \sigma \subseteq [l] ~~\mbox{and}~~ |\sigma| = r) = I.
\end{eqnarray*}
Since we have already
shown that $I \subseteq I_{\XX(l,r)}$ for all $(l,r)$, the
desired result now follows.
\end{proof}

\begin{rem} We can find an alternative proof of Lemma
\ref{productintersectionlemma}, $(iv)$, in \cite[Proposition 2.9]{GHM}.
\end{rem}

When $r = l-n+1$, Lemma
\ref{productintersectionlemma} implies that
$\XX(l,l-n+1) \subseteq \pr^n$ is a collection of $\binom{l}{n}$
points.  In this case, we can compute the corresponding Hilbert
function.

\begin{thm} \label{hfgeneric}
Let $\XX(l,l-n+1) \subset\PP^n$ be a star configuration.
Then $\mathbb{X}(l,l-n+1)$ has the Hilbert function of
${l\choose
n}$ generic points, that is,
\[HF(\mathbb{X}(l,l-n+1),t)= \dim_\mathbb{C} (S/I_{\XX(l,l-n+1)})_t =
\min\left\{{n+t\choose n},{l\choose n}\right\}.\] \noindent
Furthermore, the ideal $I_{\XX(l,l-n+1)}$ is generated by
${l\choose n-1}$ forms of degree $l-n+1.$
\end{thm}

\begin{proof}
The Hilbert function of a finite set of points $\XX$ is a
non-decreasing sequence that stabilizes at $|\XX|$, so
$HF(\XX(l,l-n+1),t) \leq \binom{l}{n}$ for all $t$. From Lemma
\ref{productintersectionlemma} (iv), because $I_{\XX(l,l-n+1)}$ is
generated in degree $l-n+1$, then $(I_{\XX(l,l-n+1)})_ t = (0)$
for all $t < l - n+1$, whence $\dim_\mathbb{C}
(S/I_{\XX(l,l-n+1)})_t = \dim_\mathbb{C} S_t = \binom{t+n}{t}$.
The conclusion now follows from the fact that $\binom{l}{n} =
\binom{t+n}{n}$ when $t = l-n$.
The second statement follows from
\cite[Proposition 4]{GO} since
$\XX(l,l-n+1)$ has the Hilbert function of ${l\choose n}$ generic
points.
\end{proof}

We can use Theorem \ref{hfgeneric} to prove the following
result.

\begin{thm}\label{l-r+1>n}
Let $L_1,\ldots,L_l$ be $l$ general linear forms of
$S = \mathbb{C}[x_0,\ldots,x_n]$ and $0 < r \leq l$.
If $l-r+1 > n$, then
\[(L_\sigma ~|~  \sigma \subseteq [l] ~~\mbox{and}~ |\sigma| = r ) =
(S_r).\]
\end{thm}

\begin{proof}  Set
$I = (L_\sigma ~|~  \sigma \subseteq [l] ~~\mbox{and}~ |\sigma| = r )$.
Because $\dim_\mathbb{C} S_r = \binom{r+n}{n}$,
the result will follow if we can find a set of $\binom{r+n}{n}$
linearly independent generators of $I$.

Let $L_1,\ldots,L_{n+r-1}$ be the first $n+r-1$ forms of
$L_1,\ldots,L_l$ (since $l > n+r-1$, there is at least one more form $L_{n+r}$).
If we set
\[I' =  (L_\sigma ~|~  \sigma \subseteq [n+r-1] ~~\mbox{and}~ |\sigma| = r ),\]
then $I'$ is the defining ideal of a star configuration
$\XX(n+r-1,r)$.  In particular, since $(n+r-1)-r+1 = n$,
$\XX(n+r-1,r)$ is a set of $\binom{n+r-1}{n}$ points.

By Theorem \ref{hfgeneric}, the ideal $I_{\XX(n+r-1,r)}$ is
generated by $\binom{n+r-1}{r} = \binom{n+r-1}{n-1}$ linearly independent
elements of degree $r$.
Let
\[A = \{ L_\sigma ~|~  \sigma \subseteq [n+r-1] ~~\mbox{and}~ |\sigma| = r \}\]
be these generators.
Now consider the set of generators of $I$ of the form:
\[B = \{L_{\tau}L_{n+r} ~|~ \tau \subseteq [n+r-1] ~~\mbox{and}~ |\tau| = r-1\}.\]
It follows that $|B| = \binom{n+r-1}{r-1}=\binom{n+r-1}{n}$. Then
$|A \cup B| = \binom{n+r-1}{n-1} + \binom{n+r-1}{n} =
\binom{n+r}{n}$.  So, we will be finished if we can show that the
elements of $A\cup B$ are linearly independent.

Suppose, for a contradiction, that there was some linear combination
\[\sum_{L_{\sigma} \in A} c_{\sigma}L_{\sigma} + \sum_{L_{\tau}L_{n+r} \in B} d_{\tau}
L_{\tau}L_{n+r} = 0\]
with $c_{\sigma},d_{\tau} \in \mathbb{C}$, not all zero.  There must be at least
one nonzero $d_{\tau}$ since all the elements of $A$ are linear
independent.  Rearranging the above equation gives:
\[\sum_{L_{\tau}L_{n+r} \in B} d_{\tau}
L_{\tau}L_{n+r} \in I' = I_{\XX(n+r-1,r)}.\]

Assume that $d_{\tau} \neq 0$.  If $\tau =
\{i_1,\ldots,i_{r-1}\}$, then $[n+r-1] \setminus \tau =
\{j_1,\ldots,j_n\}$.  Let $P$ be the point of $V(I') =
\XX(n+r-1,r)$ defined by $(L_{j_1},\ldots,L_{j_n})$. Because the
$L_i$s are general linear forms, the point $P$ does not vanish at
any of $L_{i_1},\ldots,L_{i_{r-1}},L_{n+r}$. On the other hand,
for any $\tau \neq \tau' \subseteq [n+r-1]$ with $|\tau'| = r-1$,
we must have $\tau' \cap \{j_1,\ldots,j_n\} \neq \emptyset$, and
thus $P$ vanishes at $L_{\tau'}L_{n+r}$.  We then have
\[ \left(\sum_{L_{\tau}L_{n+r} \in B} d_{\tau} L_{\tau}L_{n+r} \right)(P)
= d_{\tau}L_{\tau}(P)L_{n+r}(P) = 0.\] But since $L_{\tau}(P) \neq
0$ and $L_{n+r}(P) \neq 0$, we must have $d_{\tau} = 0$, a
contradiction.
\end{proof}

The above theorem will be key in proving Theorem
\ref{mainresult} $(3)$, i.e., when $l -r +1 >n$. When $l -r + 1
\leq n$, Question \ref{question} can be geometrically
reinterpreted:

\begin{problem}[Geometric Question]\label{geometryQ}
Let $l,r$, and $d$ be positive integers such that
$r \leq \min\{d,l\}$
and $l-r+1 \leq n$.
For a generic homogeneous
form $F \in S = \mathbb{C}[x_0,\ldots,x_n]$ of degree $d$,
is there a  star configuration $\XX(l,r)$ such that $\XX(l,r) \subseteq
V(F)$?
\end{problem}

We answer Question \ref{geometryQ} for
two trivial cases.

\begin{lem} \label{trivialcase1}
Let $l,r$, and $d$ be positive integers such that $r \leq l$
and $r \leq d$.  Furthermore, suppose that $l-r+1 =n$.
\begin{enumerate}
\item[$(i)$] If $l = n$ (and thus, $r=1$) then every generic
hypersurface of degree $d \geq 1$ contains a star configuration
$\XX(l,1)$. \item[$(ii)$] If $l = n+1$ (and thus $r=2$), then
every generic hypersurface of degree $d \geq 2$ contains a star
configuration $\XX(l,2)$.
\end{enumerate}
\end{lem}

\begin{proof}
$(i)$
Every hypersurface contains a point, which can be viewed as a $\XX(l,1)$.

$(ii)$ In this case $\XX(l,2)$ is $\binom{n+1}{n} = n+1$ points in general
linear position, and every generic hypersurface of degree $d \geq 2$ contains
such a configuration of points.
\end{proof}

We now pause and prove part of Theorem \ref{mainresult}.

\begin{proof}[{Proof of Theorem \ref{mainresult} (2), cases (i) to
(vii).}] As an opening remark, we can eliminate any tuple
$(n,l,r,d)$ that has $d < r$ or $d < l$. As mentioned in the
introduction, the statements are true for all tuples with $n=2$,
as proved in \cite{CVT};  we refer the reader to this paper for
these proofs.

We now consider the case that $n=1$, and consequently, $l-r+1 = 1$
implies that $l=r$.  Consider all the tuples of the form
$(1,l,l,d)$.  Since $l = r$, and we must have $d \geq l$, we can
omit all tuples with $d < l$.  So, it suffices to show that
Question \ref{question} has a positive answer with $n=1$ if and
only if $(n,r,l,d) = (1,l,l,d)$ with $d \geq l \geq 1$. So let us
first suppose there are general linear forms $L_1,\ldots,L_l$ such
that $F \in I = (L_\sigma ~|~ \sigma \subseteq [l] ~~\mbox{and}~~
|\sigma| = r= l ) = (L_1\cdots L_l)$.   Because $\deg F = d \geq r
=l$, we have that $(n,l,r,d) = (1,l,l,d)$ with $d \geq l =r \geq
1$. For the converse, suppose we are given a generic form $F$ of
degree $d$.  Because $F \in \mathbb{C}[x_0,x_1]$, we can factor
$F$ as $F = L_1L_2\cdots L_d$.  Because $F$ is generic, we can
assume that each $L_i$ has multiplicity one.  Since $d \geq l
\geq 1$ and because $l = r$, we can take our general linear forms
to be $L_1,\ldots,L_l$.  In this case $F \in I = (L_\sigma ~|~
\sigma \subseteq [l] ~~\mbox{and}~ |\sigma| = r ) = (L_1L_2 \cdots
L_r).$ Thus Question \ref{question} has a positive answer.

Furthermore, Lemma \ref{trivialcase1} implies that Question \ref{question}
has a positive answer if $(n,l,r,d) = (n,n,1,d)$ for all
$n \geq 3$ and $d \geq 1$ and if $(n,l,r,d) = (n,n+1,2,d)$ for
all $n \geq 3$ and $d \geq 2$.
\end{proof}

\begin{proof}[{Proof of Theorem \ref{mainresult}, (3)}] Suppose that
$(n,l,r,d) \in \mathbb{N}^4_+$ with $r \leq \min\{d,l\}$.  Suppose
that $l-r+1 >n$ and $F$ is any homogeneous form of $F$ of degree
$d$. Let $L_1,\ldots,L_l$ be any general linear forms.  Then by
Theorem \ref{l-r+1>n}, $I = (L_{\sigma} ~|~ \sigma \subseteq [l]
~\mbox{and}~~ |\sigma| = r) = (S_r)$. So, $F \in S_d \subseteq I$,
and thus, by Question \ref{algebraQ}, Question \ref{question} has
a positive answer.
\end{proof}

We continue with the proof of Theorem 1.2 at the end of the next
section.


\section{Non-existence answers}

We now give negative answers to Question \ref{question}
in a number of cases, allowing us to reduce
Question \ref{question} to one non-trivial case, which
will be studied in the remaining sections.

We first provide an asymptotic negative answer
to Question \ref{question} when $l-r+1<n$.

\begin{lem} \label{l-r+1<n}
If $l-r+1<n$ and $d\gg 0$, then the generic degree $d$
hypersurface does not contain a star configuration
$\mathbb{X}(l,r)$.
\end{lem}

\begin{proof}
Let $\PP S_d$ be the parameter space for degree $d$ hypersurfaces
in $\PP^n$. Also, let $\mathcal{H}\subset(\check{\PP^n})^l$ be the
parameter space for star configurations $\mathbb{X}(l,r)$ in
$\PP^n$.
Consider the incidence correspondence
\[\Sigma_{d,l,r}=\left\{(H,\mathbb{X}(l,r)): \mathbb{X}(l,r)\subset H\right\}
\subset\PP S_d\times\mathcal{H}\] and the natural projection maps
\[\psi_{d,l,r}:\Sigma_{d,l,r}\longrightarrow \mathcal{H}
~~\text{and}~~~ \phi_{d,l,r}:\Sigma_{d,l,r}\longrightarrow \PP
S_d.\] We have that $\phi_{d,l,r}$ is dominant if and only if
Question \ref{question} has an affirmative answer.

Using a standard fibre dimension argument,  if $d\geq l-n+1$, then we get
\[\dim\Sigma_{d,l,r}\leq \dim \mathcal{H} +
\dim_{\mathbb{C}} (I_{\mathbb{X}(l,r)})_d -1 =
\dim\mathcal{H}+{n+d\choose d}-{l\choose n} -1.\]
Hence we have that
\[\dim\Sigma_{d,l,r}-\dim\PP
S_d \leq \dim \mathcal{H} + \dim_{\mathbb{C}}
(I_{\mathbb{X}(l,r)})_d -{d+1 \choose n}=
\dim\mathcal{H}-HF(\mathbb{X}(l,r),d).\] Now $\dim\mathcal{H}=ln$
and $HF(\mathbb{X}(l,r),d)$ is an eventually positive polynomial
in $d$ of degree $n-(l-r+1)$ by Lemma
\ref{productintersectionlemma} $(ii)$. Thus, for $d \gg 0$, the
map $\phi_{d,l,r}$ cannot be dominant.
\end{proof}

We now restrict to the case that $l-r+1 = n$. In light of Question
\ref{geometryQ}, we are asking if the generic hypersurface
contains a star configuration $\XX(l,r)$.  Because $l -r +1 = n$,
$l$ determines $r$, so we will simplify our notation slightly and
write $\XX(l) \subseteq \PP^n$ for $\XX(l,r)$. We can now
eliminate ``large'' values of $d$ when $l-r+1 = n$.

\begin{thm}\label{generalcase} If $n>2$ and $l>n+2$,
then the generic degree $d$ hypersurface in $\PP^n$ does
not contain a star configuration $\mathbb{X}(l)$ for any $d$.
If $n=2$, then the generic degree $d$ plane curve does not contain a
star configuration  $\mathbb{X}(l)$ for $l>5$ and any $d$.
\end{thm}

\begin{proof}  The case $n=2$ is \cite[Theorem 3.1]{CVT},
so we only consider the case $n>2$. We use the notation of Lemma
\ref{l-r+1<n} dropping the unnecessary subindex $r$.
Using a standard fibre dimension argument,  if $d\geq l-n+1$, then
\[\dim\Sigma_{d,l}\leq \dim \mathcal{H} +
\dim_{\mathbb{C}} (I_{\mathbb{X}(l)})_d -1 =
\dim\mathcal{H}+{n+d\choose d}-{l\choose n} -1.\]
Note that we use Theorem \ref{hfgeneric} to compute
$\dim_{\mathbb{C}} (I_{\mathbb{X}(l)})_d$.  Thus, the answer
to our question is affirmative only if
$\dim\Sigma_{d,l}\geq\dim\PP S_d$, that is, only if
\begin{equation}\label{bound}
ln-{l\choose n}\geq 0.
\end{equation}
We show that \eqref{bound} does not hold if $l\geq n+3$. If
$l=n+3$, then \eqref{bound} yields
\[n(n+3)-\binom{n+3}{n}=-(n+3)\frac{n^2-3n+2}{6}\geq 0\, ,
\]
and this does not hold for $n>2$.  So suppose that $l > n + 3$. We then have
\begin{eqnarray}
ln-{l\choose n}&=&ln - \frac{l(l-1)\cdots(l-n+1)}{n!}\nonumber \\
&\geq& \frac{n+3}{n!}[n(n!)+(1-l)(l-2)(l-3)\cdots(l-n+1)] \nonumber\\
&\geq&\frac{n+3}{n!}[n(n!)+(1-l)(n+1)n\cdots4]\nonumber\\
&=&\frac{n+3}{n!}\left[n(n!)+(1-l)\frac{(n+1)!}{6}\right]\nonumber\\
&=&\frac{n+3}{n!}\left[n(n!)+(1-l)\frac{(n+1)(n!)}{6}\right]
={(n+3)}\left[n+(1-l)\frac{(n+1)}{6}\right]\geq 0 \label{bound2}.\end{eqnarray}
But \eqref{bound2} is  true only if $ \frac{7n+1}{n+1}\geq l> n+3$ and
this is a contradiction for $n>2$.
\end{proof}

We use the results of this section to continue our proof of
Theorem \ref{mainresult}.

\begin{proof}[{Proof of Theorem \ref{mainresult}, (1).}]
From Lemma \ref{l-r+1<n},
it follows that if $l-r+1 < n$ and $d\gg 0$, then the generic
degree $d$ form in $n+1$ variables cannot be written in the form
\eqref{expression}.
\end{proof}

\begin{rem} We can now assume $l-r+1 =n$.
We have already dealt with the case that $n=1$ or $n=2$.
On the other hand, if $n \geq 3$, we can eliminate tuples $(n,l,r,d)$
with $l \geq n+3$ by Lemma \ref{generalcase}.  So, we are only
left with the tuples of the form $(n,n,1,d), (n,n+1,2,d)$, and $(n,n+2,3,d)$
with $d \geq r$.  But we have already taken care of the tuples
of the form $(n,n,1,d)$ and $(n,n+1,2,d)$, so it suffices
to determine for which $d$ the Question \ref{question} has a positive
answer for $(n,n+2,3,d)$.  The remaining sections deal with this case.
\end{rem}


\section{Interlude: reformulating our question}

To complete our proof of Theorem \ref{mainresult}, it suffices to
determine which tuples of the form $(n,n+2,3,d)$ with $d \geq 3$
satisfy Question \ref{question}.  In the language of star
configurations, we wish to know which degree $d$ generic
hypersurfaces in $\pr^n$ contains a star configuration $\XX(n+2) =
\XX(n+2,3)$. We make a brief interlude to derive some technical
results, moving Question \ref{question} back and forth between
questions in algebra and questions in geometry.

We first notice the following trivial fact:

\begin{lem}\label{fsummililem} Let $\{F=0\}$ be an equation of the
degree $d$ hypersurface ${Y}\subset\PP^n$. Then $Y$ contains
a star configuration $\mathbb{X}(n+2)$ only if
there are  $L_1,\ldots,L_l,$ with $l = n+2,$ general linear forms such that
\[F=
\sum_{
\footnotesize{
\begin{array}{c}
\sigma = \{i_1,i_2,i_3\} \subseteq [n+2]\\
\end{array}}}
L_\sigma M_\sigma
\]
where the $\binom{n+2}{3}$ forms $M_{\sigma}$ have degree $d-3$.
\end{lem}

Hence, it is natural to perform the following geometric
construction. We define a map
\[\Phi_{d,l}:\underbrace{ S_1\times\cdots\times
S_1}_{l = n+2}\times\underbrace{ S_{d-3}\times\cdots\times
S_{d-3}}_{\binom{n+2}{3}}\longrightarrow S_d\]
of affine varieties
such that
\[\Phi_{d,l}\left(L_1,\ldots,L_l,M_{\{1,2,3\}},\ldots,M_\sigma,\ldots,
\ldots M_{\{n,n+1,n+2\}}\right)= \sum_{\sigma \subseteq [n+2],~~
|\sigma|=3} ^{} L_{\sigma}M_{\sigma}.\] We then rephrase our
question in terms of the map $\Phi_{d,l}$:

\begin{lem}\label{dominantmaplemma} Let $d,l=n+2$ be non-negative integers
with $d \geq l-1$. Then the following are equivalent:
\begin{itemize}
\item[$(i)$] Question \ref{question} has an affirmative answer
for $(n,l,r,d) = (n,n+2,3,d)$.
\item[$(ii)$] the map $\Phi_{d,l}$ is a dominant
map.
\end{itemize}
\end{lem}

\begin{proof}
Lemma \ref{fsummililem} proves that $(i)$ implies $(ii)$. To prove
the other direction, it is enough to show that for a generic form
$F$, the fibre $\Phi_{d,l}^{-1}(F)$ contains  a set of $l$ linear
forms defining a star configuration. More precisely, define
$\Delta\subset  S_1\times\cdots\times S_1\times
S_{d-3}\times\cdots\times S_{d-3}$ as follows:
\[\Delta = \left\{\left(L_1,\ldots,L_l,\ldots,M_\sigma,\ldots\right)~\left|
\begin{tabular}{c}
\mbox{there exists $\sigma = \{a,b,c\} \subseteq [n+2]$ such that} \\
\mbox{$L_a,L_b,L_c$ are linearly dependent}
\end{tabular} \right\}\right. .\]
Then we want to show that $\Phi_{d,l}^{-1}(F)\not\subset\Delta$.

We proceed by contradiction, assuming that the generic fibre of
$\Phi_{d,l}$ is contained in $\Delta$.  Then $\Delta$ would be a
component of the domain of $\Phi_{d,l}$. This is a contradiction as
the latter is an irreducible variety being the product of
irreducible varieties.
\end{proof}

Using the map $\Phi_{d,l}$ we can now translate Question
\ref{question} into an ideal theoretic question.

\begin{lem}\label{tgspaceideallem}
 Let $d,l=n+2$ be non-negative integers
with $d \geq l-1$.
Consider $l$ generic forms
$L_1,\ldots,L_l\in S = \mathbb{C}[x_0,\ldots,x_n]$ and
$\binom{n+2}{3}$ forms $\{M_\sigma \in S_{d-3} ~|~ \sigma \subseteq [n+2]~~\text{and}~~|\sigma|=3\}$.

Define the following $l$ forms of degree $d-1$:
\begin{eqnarray*}
Q_1 &= & \sum_{
\footnotesize{
\begin{array}{c}
\sigma \subseteq [n+2],~~1 \in \sigma
\end{array}}}
\frac{L_\sigma M_\sigma}{L_1}
=
\sum_{
\footnotesize{
\begin{array}{c}
\{a,b\} \subseteq [n+2] \setminus \{1\}
\end{array}}}
L_{a}L_{b}M_{\{1\} \cup \{a,b\}},
\\
Q_2 &= & \sum_{
\footnotesize{
\begin{array}{c}
\sigma \subseteq [n+2],~~2 \in \sigma
\end{array}}}
\frac{L_\sigma M_\sigma}{L_2}, \\
& \vdots& \\
Q_l &= & \sum_{
\footnotesize{
\begin{array}{c}
\sigma \subseteq [n+2],~~l \in \sigma
\end{array}}}
\frac{L_\sigma M_\sigma}{L_l}.
\end{eqnarray*}
With this notation, form the ideal
\[I= (L_{\sigma} ~|~ \sigma \subseteq [l]~~\mbox{and}~~ |\sigma|=3) +
(Q_1,\ldots,Q_l) \subseteq S.\]
Then the following are equivalent:
\begin{itemize}
\item[$(i)$] Question \ref{question} has an affirmative answer
for $(n,l,r,d) = (n,n+2,3,d)$;
\item[$(ii)$] $I_d=S_d$.
\end{itemize}
\end{lem}

\begin{proof}
Using Lemma \ref{dominantmaplemma} we just have to show that
$\Phi_{d,l}$ is a dominant map if and only if $I_d = S_d$.
In order to do this we will
determine the tangent space to the image of $\Phi_{d,l}$ at a
generic point $q=\Phi_{d,l}(p)$, where
$p=(L_1,\ldots,L_l,\ldots,M_\sigma,\ldots)$.  We denote
with $T_q$ this affine tangent space.

The elements of the tangent space $T_q$ are obtained as
\[\left.{d \over dt}\right|_{t=0}
\Phi_{d,l}\left(L_1+tL_1',\ldots,L_l+tL_l',M_{\{1,2,3\}}+tM_{\{1,2,3\}}',\ldots,
M_\sigma+tM_\sigma',\ldots\right)\]
\[ = \left.{d \over dt}\right|_{t=0} \sum_{\sigma=\{i,j,k\} \subseteq [l]}(L_i+tL'_i)(L_j+tL_j')(L_k+tL'_k)(M_\sigma
+tM'_\sigma) \] when we vary the forms $L'_i \in S_1$ and
$M_\sigma'\in S_{d-3}$. By a direct computation we see that the
elements of $T_q$ have the form
\[\sum_{\sigma =\{i,j,k\} \subseteq [l]}[L_i'L_{j}L_kM_{\sigma}+L_iL_j'L_kM_{\sigma}+L_iL_jL'_kM_{\sigma}+L_iL_jL_kM'_{\sigma}]\]
\[=\left(\sum_{i=1}^l L_i'\left(\sum_{i\in\sigma\subseteq [l]} {L_\sigma M_\sigma \over L_i} \right)\right)
+\left(\sum_{\sigma \subseteq [l]}L_\sigma M'_\sigma\right).\]
Since the $L_i',L_j',L_k'\in S_1$ and $M'_\sigma \in S_{d-3}$ can
be chosen freely, we obtain that $I_d=T_q$.
\end{proof}

\begin{rem}\label{cocoacomputREMARK}
Lemma \ref{tgspaceideallem} can be used to computationally provide
a positive answer for each tuple of the form $(n,n+2,3,d)$. To do
this, we proceed as follows. Given $d$ and $l=n+2$ we construct
the ideal $I$ as described above by choosing forms $L_i$ and $M_i$. We
then compute $\dim_\CC I_d$ using a computer algebra system. If
$\dim_\CC I_d$ = $\dim_\CC S_d$, then, by upper semicontinuity, we
have proved that Question \ref{question} has an affirmative answer
for that tuple $(n,n+2,3,d)$. On the other hand, if we pick
$L_i$'s and $M_i$'s such that $\dim_\CC I_d<\dim_\CC S_d$ we
cannot eliminate $(n,n+2,3,d)$ since another choice of forms may
give equality.
\end{rem}


\section{Base case:  $\XX(4)$ in $\pr^2$}

We now show that the generic degree $d \geq 3$ hypersurface of
$\PP^2$ contains a star configuration $\XX(4)$, i.e., we prove
Theorem \ref{mainresult} for $(2,4,3,d)$ for all $d \geq 3$.  Note
that this result was already proved in \cite{CVT}, but we give a new
proof that better lends itself to our induction
argument for proving that the generic degree $d \geq 3$
hypersurface of $\PP^n$ for all $n \geq 2$ contains a star
configuration $\XX(n+2)$.

We first begin with a lemma about matrices that shall prove useful:

\begin{lem}\label{matrixG}
Let $\mathcal A_r=(a_{ij})$ be a square $r\times r$ matrix
where $r>1$ and  $a_{ij}=\left\{
\begin{array}{ll}
    1 & \hbox{if $i\neq j$;} \\
    0 & \hbox{if $i=j$.} \\
\end{array}
\right.$, i.e.,
\begin{equation}\label{matrix}
\mathcal A_r=\left(
\begin{array}{ccccccc}
  0 & 1 & 1&\ldots &  1 \\
  1 & 0& 1&\ldots & 1 \\
   1&1&\ddots&\ddots&1\\
  \vdots & & \ddots &\ddots &1 \\
  1 & 1 &\ldots & 1&0 \\
\end{array}
\right).\end{equation}
Then $\mathcal A_r$ has maximal rank.
\end{lem}

\begin{proof} Consider the square $r\times r$ matrix
$U=\mathcal {A}_r+I_r$, i.e. this is a matrix
where every element is equal to one. Clearly ${\rm rk}(U)=1$, and
the eigenvalues are $r$, with multiplicity one, and $0$, with
multiplicity $(r-1)$. Suppose that ${\rm det}(\mathcal{A}_r)=0$. Then
there exists an eigenvector $v \neq 0$ such that $\mathcal{A}_r
v=0$, and hence, $(\mathcal {A}_r+I_r)v = \mathcal{A}_rv + I_rv = v$,
i.e., $U$ should have $1$ as an eigenvalue.
From this contradiction, we deduce that ${\rm rk}(\mathcal{A}_r)=r$.
\end{proof}

We will use the following notation in the proof given below. Let $d \geq 3$
be a non-negative integer, let
 $L_1,\ldots,L_4\in S_1$  be four generic linear
forms in $S = \mathbb{C}[x_0,x_1,x_2]$ and consider any six forms
\[\{M_\sigma \in S_{d-3} ~|~ \sigma \subseteq [4] ~~\text{and}~~ |\sigma| = 3\}.\]
We will abuse notation and write $M_{ijk}$ for $M_{\{i,j,k\}}$.
Using these forms, we define the follow four forms of degree
$d-1$:
\begin{eqnarray*}
Q_1 &= &M_{123}L_{2}L_3 + M_{124}L_2L_4 + M_{134}L_3L_4 \\
Q_2 &= &M_{123}L_1L_3 + M_{124}L_1L_4 + M_{234}L_3L_4 \\
Q_3 & =&M_{123}L_1L_2 + M_{134}L_1L_4+ M_{234}L_2L_4 \\
Q_4 & =&M_{124}L_1L_2 + M_{134}L_1L_3 + M_{234}L_2L_3.
\end{eqnarray*}
With this notation, we form the ideal
\begin{equation}\label{ideal}
I=(L_{1}L_2L_3,\ldots, L_2L_{3}L_4,Q_1,\ldots,Q_4)\subset S.
\end{equation}
Then for  $d \geq 3$, we give an affirmative answer to Question
\ref{question}:

\begin{thm}\label{l=4n=2}
The generic degree $d\geq 3$ curve in $\pr^2$ contains a
$\mathbb{X}(4)$.
\end{thm}

\begin{proof}
Our strategy is to use Lemma \ref{tgspaceideallem} to show that the
ideal $I$ of \eqref{ideal} has the
property that $I_{d} = S_d$.  In particular, given $4$ generic
linear forms $L_1,\ldots,L_4$, we need to pick forms
$M_{\sigma}$ with $\sigma \subseteq [4]$ and $|\sigma| =3$ so the
ideal \eqref{ideal} has this desired property.

Because the generators
$\{L_{\sigma} ~|~ \sigma \subseteq [4] ~\text{and}~ |\sigma|=3\}$
are the generators of a star configuration $\XX(4)$, we know by
Theorem \ref{hfgeneric} that for all $d \geq 3$
\[\dim_{\mathbb{C}} (S/(L_{123},\ldots,L_{234}))_d = 6.\]
If we set $A =  S/(L_{123},\ldots,L_{234})$, it therefore suffices
to find 6 linear independent elements in
$I/(L_{123},\ldots,L_{234})$ of degree $d$.  We will prove
that the equivalence classes of the following 6 elements in $A$
are linearly independent
\begin{equation}\label{formsl=5}
L_2Q_3,L_1Q_2,L_3Q_1,L_4Q_1,L_4Q_2,L_4Q_2,L_4Q_3
\end{equation}
for a generic choice of the forms $M_\sigma$ with $\deg M_\sigma=d-3$.
As noted by Remark \ref{cocoacomputREMARK}, it is
enough to show these forms are linearly independent for a special choice of forms for $M_{\sigma}$.

We first construct an evaluation table.  For each
$\tau =\{i,j\} \subseteq [4]$, let
\[p_{r,s} := V(L_i) \cap V(L_j) ~~~\mbox{where $\{r,s\} \cup \tau = [4]$}\]
denote one of the six points of $\XX(4)$.
We construct the following
evaluation table
where entry $(i,j)$ is formed by evaluating
the polynomial labeling column $j$ at the point
labeling row $i$.
\begin{equation}\label{evl}
\begin{array}{c|cccccc}
       & L_3Q_ 1       & L_1Q_2          & L_2Q _3         & L_4Q_1          & L_4Q_2         & L_4Q_3          \\
\hline
p_{1,2}& 0              & 0              & M_{123}L_1L_2^2    & 0              & 0              & 0               \\
p_{1,3}& 0              & M_{123}L_1L_3^2 & 0                 & 0              & 0              & 0               \\
p_{2,3}& M_{123}L_2L_3^2 &   0            & 0                 & 0              & 0              & 0               \\
p_{1,4}& 0              & M_{124}L_1^2L_4  & 0                & 0              & M_{124}L_1L_4^2  & M_{134}L_1L_4^2  \\
p_{2,4}& 0              & 0              & M_{234}L^2_2L_4    & M_{124}L_2L_4^2 & 0               & M_{234}L_2L_4^2  \\
p_{3,4}& M_{134}L^2_3L_4 & 0              & 0                 & M_{134}L_3L_4^2 & M_{234}L_3L_4^2    &0               \\
\end{array}\end{equation}
For example, the entry $(4,2)$ is the polynomial
$L_1Q_2$ evaluated at $p_{1,4}$, that is
\begin{eqnarray*}
L_1Q_2(p_{1,4}) & = & (M_{123}L_1^2L_3 + M_{124}L_1^2L_4 + M_{234}L_1L_3L_4)(p_{1,4})  \\
&= &  (M_{124}L_1^2L_4)(p_{1,4}) ~~\mbox{since $L_3(p_{1,4}) = 0$.}
\end{eqnarray*}
\normalsize With a slight abuse of notation we adopt the following
convention that if in the row indexed by  $p_{i,j}$ we write
$M_{ijk}L^c_aL^d_b$, then  this a shorthand form for
$(M_{ijk}L^c_aL^d_b)(p_{i,j})$.

Observe that the evaluation matrix (\ref{evl}) holds for any choice of
$M_\sigma \in S_{d-3}$.  For each $d \geq 3$, we want to show one
can pick specific $M_{\sigma}$'s so that this matrix has rank 6.  It
would then follow that the forms \eqref{formsl=5} are linearly
independent in $A$, and the conclusion follows.

Note that for any nonzero choice $M_{123}$, the first three rows of this matrix are linear independent.  We will be finished
if we can show that the submatrix formed by the last three rows and three columns has maximal rank, i.e,
we can find choices for $M_{124}, M_{134}$, and $M_{234}$ that make the matrix
\begin{equation}\label{matrix2}
\begin{array}{c|ccc}
       & L_4Q_1          & L_4Q_2         & L_4Q_3          \\
\hline
p_{1,4}& 0 & M_{124}L_1L_4^2  & M_{134}L_1L_4^2  \\
p_{2,4}& M_{124}L_2L_4^2 & 0               & M_{234}L_2L_4^2  \\
p_{3,4}& M_{134}L_3L_4^2 & M_{234}L_3L_4^2    &0               \\
\end{array}\end{equation}
have rank three.

When $d = 3$, we set $M_{\sigma} =1 $ when $4 \in \sigma$.  We can therefore factor
the above matrix as
\[
\begin{bmatrix}
L_1L_4^2 & 0 & 0 \\
0       & L_2L_4^2 & 0 \\
0 & 0 & L_3L_4^2
\end{bmatrix}
\begin{bmatrix}
0 & 1 & 1 \\
1 & 0 & 1 \\
1 & 1 & 0
\end{bmatrix}.\]
The first matrix is clearly invertible, and the second matrix is invertible by Lemma \ref{matrixG}.  Thus the matrix (\ref{matrix2}) has rank three, and thus the entire
evaluation matrix (\ref{evl}) has maximal rank.

When $d > 3$, we set
\[
\begin{array}{rclrcrrr}
M_{124} & = & L_2^{d-3} &  &  M_{234} = L_3^{d-3}&& M_{134} = L_4^{d-3}. \\
\end{array}
\]
When we use this choice of $M_{\sigma}$'s, the evaluation matrix (\ref{matrix2})  given above becomes
\[
\begin{array}{c|ccc}
       & L_4Q_1          & L_4Q_2         & L_4Q_3          \\
\hline
p_{1,4}& 0     & 0       & \star  \\
p_{2,4}& \star & 0       & 0 \\
p_{3,4}& \star & \star   & 0               \\
\end{array}\]
where $\star$ represents a non-zero value.
But then it is immediate that this matrix has
rank three, and thus the entire matrix (\ref{evl}) has maximal rank.
\end{proof}


\section{Induction Step: $\XX(n+2)$ in $\pr^n$}

We now prove the general situation.

\begin{thm}\label{l=n+2}
The generic degree $d\geq 3$ hypersurface of $\pr^n$ contains a
$\mathbb{X}(n+2)$.
\end{thm}

\begin{proof} We work by induction on $n.$ If $n=2$, then the result
is true by Theorem \ref{l=4n=2}.

To better understand the induction step, we will show how we pass
from the case $n=2$ to $n=3$. In $\pr^3$, $l=5$, and thus $\XX(5)$ contains $10$
points $\{p_{i,j} ~|~ 1 \leq i < j \leq 5\}$.  In particular
for each $\tau = \{i_1,\ldots,i_3\} \subseteq [5]$ with $|\tau| =3$,
\[p_{r,s} := V(L_{i_1}) \cap \cdots \cap V(L_{i_3}) ~~~\mbox{where $\{r,s\} \cup \tau = [5]$
for general linear forms $L_1,\ldots,L_5$.}\]

For each $d\geq 3$, we construct an evaluation matrix $\mathcal {M}_3$
of size $10\times 10$ in the following way.  Let $Q_1,\ldots,Q_5$ be the forms
constructed from $L_1,\ldots,L_5$ as in Lemma \ref{tgspaceideallem}.  Our evaluation
matrix is then:
\[\mathcal{M}_3=
\begin{array}{c|ccc|cccc}
       & L_3Q_1 & \cdots & L_4Q_3 &L_5Q_1 & L_5Q_2 & \cdots & L_5Q_4\\
\hline
p_{1,2} &        &        &       &       &        &       &  \\
\vdots &        & \overline{\mathcal{M}}_2 &&&    &  {\bf 0}    &  \\
p_{3,4} &        &        &       &       &        &       &  \\
\hline
p_{1,5} &        &        &       &        &       &        & \\
\vdots &        & \mathcal{F} &  &        &       &  \mathcal{G}      & \\
p_{4,5} &        &        &       &        &       &        &
\end{array}
\]
where the matrix $\overline{\mathcal{M}_2}$ is formally
the same as the $6 \times 6$ matrix constructed in the proof
of Theorem \ref{l=4n=2}, ${\bf 0}$ is the $6 \times 4$ zero matrix,
$\mathcal{F}$ is a $4 \times 6$ matrix, and $\mathcal{G}$
is a $4 \times 4$ matrix.  It should be clear that the top right block
is the zero matrix since each point $p_{1,2},\ldots,p_{3,4}$ vanishes
on the line $V(L_5)$.

As in Theorem \ref{l=4n=2}, we need to show that we can pick the
$M_{\sigma}$'s that
appear in the construction of
$Q_1,\ldots,Q_5$ so that the above evaluation matrix
has ${\rm rk}(\mathcal{M}_3) = 10$.  By induction, we can find $M_{\sigma}$'s
with $\sigma \subseteq [4]$ and $|\sigma| = 3$ so that the matrix
$\overline{\mathcal{M}}_2$ has maximal rank.  We will therefore finish the proof
for the $n=3$ case if we can show that the matrix $\mathcal{G}$ has rank 4.

The matrix $\mathcal G$ is a $4\times 4$ matrix of type
\begin{center}
\begin{tabular}{c|cccccccccc|cccccc} &  $L_5Q_1$ & $L_5Q_2$ & $L_5Q_3$ & $L_5Q_4$ \\
\hline
  $p_{15}$ &0&$M_{125}L_1L_5^2$&$M_{135}L_1L_5^2$&$M_{145}L_1L_5^2$\\
 $p_{25}$ &$M_{125}L_2L_5^2$ & 0& $M_{235}L_2L_5^2$ & $M_{245}L_2L_5^2$ \\
 $p_{35}$ &$M_{135}L_3L_5^2$ &$M_{235}L_3L_5^2$&0&$M_{345}L_3L_5^2$ \\
 $p_{45}$  & $M_{145}L_4L_5^2$ & $M_{245}L_4L_5^2$ & $M_{345}L_4L_5^2$ &0
\end{tabular}
\end{center}
Again, as in the proof of Theorem \ref{l=4n=2}, we write $M_{ijk}L_a^bL_c^d$ to mean
the value of $M_{ijk}L_a^bL_c^d(p_{r,s})$.
Note that no $M_{\sigma}$ with $\sigma \subseteq [4]$ appears in the above matrix,
so fixing these values in $\overline{\mathcal{M}}_2$ has no impact in $\mathcal{G}$.

When $d =3$, we set each $M_{ij5} = 1$ when defining $Q_1,\ldots,Q_5$.  In this case,
we can factor $\mathcal{G}$ as
\[
\begin{bmatrix}
L_1L_5^2 & 0 & 0 & 0 \\
0 & L_2L_5^2 & 0 & 0 \\
0 & 0 & L_3L_5^2 & 0 \\
0 & 0 & 0 & L_4L_5^2 \\
\end{bmatrix}
\begin{bmatrix}
0 & 1 & 1 & 1 \\
1 & 0 & 1 & 1 \\
1 & 1 & 0 & 1 \\
1 & 1 & 1 & 0 \\
\end{bmatrix}.
\]
Both matrices are invertible (we are using Lemma \ref{matrixG} for the second matrix),
so the  matrix $\mathcal{G}$ is invertible, and thus has maximal rank.

When $d  > 3$, we set
\[
\begin{array}{rclrl}
M_{125} & = & L_2^{d-3} &  &  M_{235}  =  L_3^{d-3}\\
M_{135} & = & L_3^{d-3} &  &  M_{245} = L_4^{d-3} \\
M_{145} & = & L_5^{d-3} &  &  M_{345} = L_4^{d-3}. \\
\end{array}
\]
With these choices, the evaluation matrix $\mathcal{G}$ becomes:
\begin{center}
\begin{tabular}{c|cccccccccc|cccccc} &  $L_5Q_1$ & $L_5Q_2$ & $L_5Q_3$ & $L_5Q_4$ \\
\hline
 $p_{16}$ &0       & 0       &  0       &$\star$ \\
 $p_{26}$ &$\star$ & 0       &  0       & 0      \\
 $p_{36}$ &$\star$ & $\star$ &  0       & 0      \\
 $p_{46}$ &$\star$ & $\star$ &  $\star$ & 0
\end{tabular}
\end{center}
where $\star$ is a non-zero value.  In this form, it is clear that the matrix
has maximal rank.

We now describe the general induction step.  That is, suppose that
$d\geq 3$ and that the theorem is true for $\pr^n$. We prove the
statement for $\pr^{n+1}$.

Let $L_1,\ldots,L_{n+3}$ be $n+3$ generic linear forms.  For each
$\tau = \{i_1,\ldots,i_{n+1}\} \subseteq [n+3]$ with $|\tau| = n+1$, we set
\[p_{r,s} := V(L_{i_1}) \cap \cdots \cap V(L_{i_{n+1}}) ~~~\mbox{where $\{r,s\} \cup \tau = [n+3]$.}\]
To finish the proof, it suffices to show that we can find choices
for $M_\sigma$ as $\sigma \subseteq [n+3]$ with $|\sigma| = 3$
so that the evaluation matrix
\[\mathcal{M}_{n+1}=
\begin{array}{c|ccc|cccc}
       & L_5Q_1 & \cdots & L_{n+2}Q_{n+1} &L_{n+3}Q_1 & L_{n+3}Q_2 & \cdots & L_{n+3}Q_{n+2}\\
\hline
p_{1,2} &        &        &       &       &        &       &  \\
\vdots &        & \overline{\mathcal{M}}_n &&&    &  {\bf 0}    &  \\
p_{n+1,n+2} &        &        &       &       &        &       &  \\
\hline
p_{1,n+3} &        &        &       &        &       &        & \\
\vdots &        & \mathcal{F} &  &        &       &  \mathcal{G}      & \\
p_{n+2,n+3} &        &        &       &        &       &        &
\end{array}
\]
has maximal rank.  Here, $\overline{\mathcal{M}}_n$ is formally the same matrix as
$\mathcal{M}_n$, the matrix ${\bf 0}$ is an appropriate sized zero matrix, $\mathcal{F}$
is a $(n+2) \times \binom{n+2}{2}$ matrix, and $\mathcal{G}$ is a $(n+2) \times (n+2)$
matrix of the form:
\footnotesize
\begin{center}
\begin{tabular}{c|ccccc} &  $L_{n+3}Q_1$ & $L_{n+3}Q_2$ & $\cdots$ & $L_{n+3}Q_{n+1}$ & $L_{n+3}Q_{n+2}$  \\
\hline
  $p_{1,n+3}$ &0&$M_{1,2,n+3}L_1L_{n+3}^2$& $\cdots$ &$M_{1,n+1,n+3}L_1L_{n+3}^2$&$M_{1,n+2,n+3}L_1L_{n+3}^2$ \\
 $p_{2,n+3}$ &$M_{1,2,n+3}L_2L_{n+3}^2$ & 0& $\cdots$ & $M_{2,n+1,n+3}L_2L_{n+3}^2$ & $M_{2,n+2,n+3}L_2L_{n+3}^2$\\
 $\vdots$ &   $\vdots$ &   & $\ddots$ &   & $\vdots$ \\
 $p_{n+1,n+3}$  & $M_{1,n+1,n+3}L_{n+1}L_{n+3}^2 $ & $M_{2,n+1,n+3}L_{n+1}L_{n+3}^2$ & $\cdots$ &0& $M_{n+1,n+2,n+3}L_{n+1}L_{n+3}^2$\\
 $p_{n+2,n+3}$ & $M_{1,n+2,n+3}L_{n+2}L_{n+3}^2$ & $M_{2,n+2,n+3}L_{n+2}L_{n+3}^2$ & $\cdots$ & $M_{n+1,n+2,n+3}L_{n+2}L_{n+3}^2$ &0
\end{tabular}
\end{center}
\normalsize

By induction, we can find $M_\sigma$ with $\sigma \subseteq [n+3]$
and $|\sigma| = 3$, and $n+3 \not\in \sigma$ so that the matrix
$\overline{\mathcal{M}_{n}}$ has maximal rank.  It remains to show
that $\mathcal{G}$ has maximal rank.

As in the case $n=3$, when $d=3$, we set every $M_\sigma = 1$ when
$\sigma \subseteq [n+3]$ with $|\sigma| = 3$, and $n+3 \in
\sigma$.  Using Lemma \ref{matrixG}, we can show that
$\mathcal{G}$ has maximal rank.  When $d > 3$, we set
\[M_{i,j,n+3} = L_j^{d-3}  ~~~\mbox{for all $\sigma \subseteq [n+3]$ with $|\sigma| = 3$, and
$n+3 \in \sigma$}\]
{\bf except} for $M_{1,n+2,n+3}$, which we set to $M_{1,n+2,n+3} = L_{n+3}^{d-3}$.  The evaluation
matrix $\mathcal{G}$ then becomes
\begin{center}
\begin{tabular}{c|ccccc} &  $L_{n+3}Q_1$ & $L_{n+3}Q_2$ & $\cdots$ & $L_{n+3}Q_{n+1}$ & $L_{n+3}Q_{n+2}$  \\
\hline
  $p_{1,n+3}$ &$0$&$0$& $\cdots$ &$0$&$\star$ \\
 $p_{2,n+3}$ & $\star$ & $0$& $\cdots$ & $0$ & $0$\\
 $\vdots$ &   $\vdots$ &   & $\ddots$ &   & $\vdots$ \\
 $p_{n+1,n+3}$  & $\star$ & $\star$ & $\cdots$ &$0$& $0$\\
 $p_{n+2,n+3}$ & $\star$ & $\star$ & $\cdots$ & $\star$ &0
\end{tabular}
\end{center}
where $\star$ represents a non-zero value.  Because it is clear that this
matrix will have rank $n+2$, this completes the proof.
\end{proof}

We are now able to complete the proof the main theorem:

\begin{proof}[{Proof of Theorem \ref{mainresult}}, (2), case (viii).] By Theorem \ref{l=4n=2} and Theorem
\ref{l=n+2}, Question \ref{question} holds for all tuples of the
form $(n,n+2,3,d)$ with $d \geq 3$.
\end{proof}


\end{document}